\documentclass[11pt]{article}
\usepackage[margin=1in]{geometry}
\usepackage{amsmath,amssymb,amsthm,mathtools}
\usepackage{microtype}
\usepackage{enumitem}
\usepackage{hyperref}
\usepackage[nameinlink,capitalise]{cleveref}
\hypersetup{hidelinks}
\numberwithin{equation}{section}
\usepackage{dirtytalk}

\newtheorem{theorem}{Theorem}[section]
\newtheorem{lemma}[theorem]{Lemma}
\newtheorem{proposition}[theorem]{Proposition}
\newtheorem{corollary}[theorem]{Corollary}

\newtheorem{remark}[theorem]{Remark}

\newcommand{\R}{\mathbb{R}}
\newcommand{\E}{\mathbb{E}}
\newcommand{\1}{\mathbf{1}}
\newcommand{\ip}[2]{\left\langle #1,#2\right\rangle}
\newcommand{\Gap}{\operatorname{Gap}}

\title{Dynamic Averaging on Regular Graphs}
\author{Dean Kraizberg\thanks{School of Mathematical Sciences, Tel Aviv University, deank@mail.tau.ac.il}}

\begin{document}
\maketitle

\begin{abstract}
We study a dynamic averaging process on finite regular graphs with bounded, time-varying load arrivals. At each discrete time $t$, an edge is chosen uniformly at random, a load $0\le w_t \le 1$ is introduced, and the total load of its two endpoints together with $w_t$ is divided equally between them. Starting from the flat configuration, we obtain a pairwise concentration bound governed by the effective resistance between vertices and use generic chaining to derive a general upper bound on the expected gap between the largest and smallest loads.

As a consequence, we show that every $d$-regular graph has expected gap $O_d(\sqrt n)$, uniformly in time and over all deterministic arrival sequences. Applying our general bound to the discrete two-dimensional torus yields the sharp $O(\log n)$ upper bound, improving the best previously known bound. For the cycle, whenever the arriving loads are bounded away from zero, we prove that the expected gap is $\Omega(\sqrt n)$ for all sufficiently large times. Together with our upper bound, this confirms the conjecture of Alistarh, Nadiradze, and Sabour that the expected gap on the cycle is of order $\sqrt n$.
\end{abstract}

\section{Introduction}
Randomized load balancing has its origins in the classical balls-into-bins model. In the classical \(d\)-choice process, \(m\) balls are inserted sequentially into \(n\) bins, and each ball is placed in the least loaded among \(d\) bins chosen uniformly at random. Azar, Broder, Karlin, and Upfal~\cite{AzarBroderKarlinUpfal} proved that, when \(d\ge2\) and \(m=n\), the maximum load is \(\frac{\log\log n}{\log d}+O(1)\) with high probability, a dramatic improvement over the \(\Theta(\frac{\log n}{\log\log n})\) maximum load obtained with a single random choice.

\medskip

Berenbrink, Czumaj, Steger, and V\"ocking~\cite{BerenbrinkCzumajStegerVocking} extended this analysis to the heavily loaded regime \(m=\omega(n)\), showing that the imbalance remains essentially independent of the average load. Peres, Talwar, and Wieder~\cite{PeresTalwarWieder} subsequently introduced graphical balanced allocations, in which the bins are the vertices of a graph and a uniformly chosen edge determines the two possible destinations of each arriving ball. For example, their general bound yields an
\(O(n\log n)\) expected gap when applied to the cycle.

\medskip

A different but closely related line of work concerns averaging
processes, in which neighboring vertices repeatedly replace their loads by their common average. In the static setting, where the initial load is fixed and no new load enters the system, the convergence of the continuous averaging process is closely related to the mixing properties of the underlying graph (see, for example, Aldous and Lanoue \cite{AldousLanoue}). Sauerwald and Sun \cite{SauerwaldSun} studied discrete variants of this process and obtained strong discrepancy bounds for a wide range of network topologies.

\medskip

The dynamic version of the problem combines these two features: load is continually introduced while local averaging attempts to keep the system balanced. Alistarh, Nadiradze, and Sabour~\cite{AlistarhNadiradzeSabour} initiated the study of this process on the cycle. For weighted arrivals with bounded second moment, they obtained an \(O(\sqrt n\log n)\) upper bound on the expected gap, together with a lower bound of order \(n\) for its second moment. Based on this and numerical evidence, they conjectured that the expected gap on the cycle should in fact be of order \(\sqrt n\).

\medskip

Berenbrink, Hintze, Hosseinpour, Kaaser, and Rau~\cite{BerenbrinkHintzeHosseinpourKaaserRau} later studied dynamic averaging on general graphs and considered several variants of the process. Their analysis connected the behavior of the process to the effective resistance, hitting times, and related random-walk quantities of the underlying graph. In particular, their bounds improved the expected gap estimate on the cycle to \(O(\sqrt{n \log n})\). They provided general results for other graph families, including the two-dimensional torus with expected gap $O(\log ^{3/2}n)$ and expanders with expected gap $O(\log n)$. 

\medskip

In this work we study the averaging process on finite regular graphs, allowing an arbitrary deterministic sequence of bounded arriving loads. Our starting point is a pairwise tail bound for the difference between the loads of two vertices. The scale of this estimate is determined by the effective resistance between the vertices, together with the largest effective resistance of an edge. This leads to pairwise bounds that depend on the effective resistance structure of the underlying graph.

\medskip

To pass from these pairwise estimates to the maximum difference over all vertices, we use Talagrand's generic-chaining functionals~\cite{TalagrandOrig,TalagrandGamma}, in the mixed-tail form developed by Dirksen~\cite{DirksenGenericChaining}. This leads to a general upper bound expressed through the effective resistance metric of the graph. In addition to applying to several examples, this form of the bound makes it possible to use known results on effective resistance and cover times to obtain discrepancy estimates for broad classes of regular graphs.

\medskip

We dive into two examples of graph families. Namely, the cycle and the discrete two-dimensional torus, for which our result gives \(O(\sqrt n)\) and \(O(\log n)\) upper bounds on the expected gap, respectively. Adapting the lower-bound argument of Bansal and Feldheim~\cite{BansalFeldheim} shows that the bound for the torus is sharp up to constants. More generally, the relation between the generic-chaining functional and the cover time of a random walk, together with the classical cover-time bound for regular graphs, yields an \(O_d(\sqrt n)\) bound for every \(d\)-regular graph.

\medskip

Finally, we return to the cycle and prove the stronger matching lower bound there. The argument combines the second-moment estimate of Alistarh, Nadiradze, and Sabour with a fourth-moment estimate obtained from our pairwise concentration bound. This shows that, whenever the arriving loads are bounded away from zero, the expected gap on the cycle is of order \(\sqrt n\), thereby confirming the conjecture of Alistarh, Nadiradze, and Sabour.

\section{Main Results}
\label{sec:main-results}

We begin by introducing some notation that will be used throughout this work. Let \(G=(V,E)\) be a finite connected \(d\)-regular graph with \(|V|=n.\) We denote \(e_v\in\mathbb R^V\) the standard basis vector corresponding to \(v\in V\), and \(\1\in\mathbb R^V\) the vector whose all entries are ones.
Fix a deterministic sequence \((w_t)_{t\ge0}\) satisfying\footnote{By rescaling, one may equivalently allow any nonnegative bounded sequence of arriving loads.}
\(
0\le w_t\le1
\)
for every \(t\ge0\). Starting from the flat configuration \(x_i(0)=0\), the process evolves as follows. 

\medskip

At time \(t\), choose an edge \(e=\{u,v\}\) uniformly from \(E\), introduce load \(w_t\), and replace the two endpoint loads by
\[
x_u(t+1)
=
x_v(t+1)
=
\frac{x_u(t)+x_v(t)+w_t}{2}.
\]
Every other coordinate remains unchanged. Since exactly \(w_t\) units of load are added in round \(t\), the mean load at time \(t\) is
\(
\frac1n\sum_{s<t}w_s.
\)
We define the centered load vector by
\[
X_i(t)
:=
x_i(t)-\frac1n\sum_{s<t}w_s.
\]
Then
\(
\sum_{i\in V}X_i(t)=0,
\)
and 
\(
\Gap(t)
:=
\max_{i\in V}X_i(t)-\min_{i\in V}X_i(t)
=
\max_{i\in V}x_i(t)-\min_{i\in V}x_i(t).
\)

\medskip

Let \(\mathcal L\) be the Laplacian matrix of the graph $G$. Namely,
\[
\mathcal L
:=
\sum_{\{u,v\}\in E}
(e_u-e_v)(e_u-e_v)^{\mathsf T},
\]
and let \(\mathcal L^+\) denote its Moore--Penrose pseudoinverse. Since \(G\) is connected,
\(
\ker(\mathcal L)=\operatorname{span}\{\1\},
\)
and \(\mathcal L\mathcal L^+=\mathcal L^+\mathcal L\) is the orthogonal projection onto the zero-sum subspace \(\1^\perp\).

\medskip

For \(u,v\in V\), the \emph{effective resistance} between \(u\) and \(v\) is
\[
R_{\mathrm{eff}}(u,v)
:=
(e_u-e_v)^{\mathsf T}
\mathcal L^+
(e_u-e_v).
\]
See~\cite{KleinRandic,LyonsPeres} for further background on effective resistance and its electrical interpretation. We use the resistance metric
\[
\rho(u,v):=\sqrt{dR_{\mathrm{eff}}(u,v)},
\]
and denote the scaled largest effective resistance of an edge by
\[
\Re_G
:=
 d\max_{\{u,v\}\in E}R_{\mathrm{eff}}(u,v).
\]

\begin{remark} \label{rem:sameER}
    The parameter \(\Re_G\) is always bounded by $d$. indeed, the edge \(\{u,v\}\) itself is a unit-resistance path, so \(R_{\mathrm{eff}}(u,v)\le1\) and hence \(\Re_G\le d\). Whenever the effective resistances of the edges are equal (for example when \(G\) is additionally edge-transitive), we can compute explicitly the value of the parameter 
\[
\Re_G
=\frac 2 n \sum_{\{u,v\}\in E}R_{\mathrm{eff}}(u,v)
=\frac 2 n\sum_{\{u,v\}\in E}
(e_u-e_v)^{\mathsf T}\mathcal L^+(e_u-e_v)=
\frac 2 n\operatorname{tr}\left(
\mathcal L^+
\sum_{\{u,v\}\in E}
(e_u-e_v)(e_u-e_v)^{\mathsf T}
\right)=
\]
\[
\frac 2 n \operatorname{tr}(\mathcal L^+\mathcal L)
= \frac{2(n-1)}{n}.
\]
\end{remark}

We present now the main results of this work. Our first main result gives pairwise concentration with respect to the effective resistance metric on the graph.

\begin{theorem}[Pairwise tail bound]
\label{thm:pairwise-tail}
There exists a universal constant \(c>0\) such that, for every finite connected \(d\)-regular graph \(G=(V,E)\), every deterministic sequence \(0\le w_t\le1\), every \(t\ge0\), every \(u,v\in V\), and every \(x\ge0\),
\[
\Pr\left(
|X_v(t)-X_u(t)|\ge x
\right)
\le
2\exp\left(
-c\min\left\{
\frac{x^2}{dR_{\mathrm{eff}}(u,v)+\Re_G+2},
\frac{x}{\Re_G+1}
\right\}
\right).
\]
\end{theorem}

We present the definition of the \emph{generic chaining functional}, originally introduced by Talagrand~\cite{TalagrandOrig, TalagrandGamma}. A sequence \(\mathcal A=(\mathcal A_k)_{k\ge0}\) of partitions of a finite set \(V\) is called admissible if
\[
|\mathcal A_0|=1,
\qquad \text{and} \qquad
|\mathcal A_k|\le2^{2^k}
\]
for every \(k\ge1\). If \(A_k(v)\) is the member of \(\mathcal A_k\) containing \(v\), then for a semimetric \(r\) on \(V\) and \(\alpha\ge1\) define
\[
\gamma_\alpha(V,r)
:=
\inf_{\mathcal A}
\sup_{v\in V}
\sum_{k\ge0}
2^{k/\alpha}
\operatorname{diam}(A_k(v),r),
\]
where the infimum is taken over all admissible sequences. The second main result shows that the expected gap is controlled by $\gamma_2(V,\rho)$.

\begin{theorem}[Expected gap bound]
\label{thm:general-upper}
There exists a universal constant \(C>0\) such that, for every finite connected \(d\)-regular graph \(G=(V,E)\), every deterministic sequence \(0\le w_t\le1\), and every \(t\ge0\),
\[
\E[\Gap(t)]
\le
C\left(
\gamma_2(V,\rho)
+
(\Re_G+1)\log(2n)
\right).
\]
In particular, if
\(
R_{\max}(G)
=
\max_{u,v\in V}R_{\mathrm{eff}}(u,v),
\)
then
\[
\E[\Gap(t)]
\le
C\left(
\sqrt{dR_{\max}(G)\log(2n)}
+
(\Re_G+1)\log(2n)
\right).
\]
\end{theorem}

The following is an immediate consequence of Theorem~\ref{thm:general-upper}, and it gives an improvement over the existing upper bounds on the expected gap of the averaging process performed on the cycle and the torus.

\begin{corollary}
\label{cor:cycle-upper}
Let $C_n$ be the cycle graph with $n$ vertices. There exists a universal constant \(C>0\) such that, for every \(n\ge3\), every deterministic sequence \(0\le w_t\le1\), and every \(t\ge0\), the dynamic averaging process on \(C_n\) satisfies
\[
\E[\Gap(t)]
\le
C\sqrt n.
\]
\end{corollary}

\begin{corollary}
\label{cor:torus}
Let \(\mathbb T_m^2 = (\mathbb Z/m\mathbb Z)^2\) be the discrete two-dimensional torus with \(n=m^2\) vertices. There exists a universal constant \(C>0\) such that, for every deterministic sequence \(0\le w_t\le1\) and every \(t\ge0\), the dynamic averaging process on \(\mathbb T_m^2\) satisfies
\[
\E[\Gap(t)]
\le
C\log(n).
\]
\end{corollary}

Building on the connection between effective resistance and cover time of random walks on the graph~\cite{DingLeePeres,ChandraResistance}, we show that the expected gap is bounded by $O(\sqrt n)$ term uniformly for every regular graph with $n$ vertices.

\begin{theorem}
\label{thm:bounded-degree-regular}
Let \(G=(V,E)\) be any finite connected \(d\)-regular graph with \(|V|=n.\)
Then, for every deterministic arrival sequence \(0\le w_t\le1\) and every \(t\ge0\),
\[
\E[\Gap(t)]
\le
C_d\sqrt n,
\]
where $C_d$ is a constant depending only on $d$.
\end{theorem}

Whenever the arriving loads are bounded away from zero, an identical argument to that of Bansal and Feldheim~\cite{BansalFeldheim} gives a universal lower bound on the expected gap for regular graphs. 

\begin{theorem}~\cite[Theorem~1.2]{BansalFeldheim}]
\label{thm:main-lower}
Let \(G=(V,E)\) be a finite connected \(d\)-regular graph with \(|V|=n\). For every \(\delta>0\), there exists a constant \(c_\delta>0\) such that for every deterministic sequence \(0<\delta\le w_t\le  1\),
\[
\sup_{t\ge0}\E[\Gap(t)]
\ge
c_{\delta}\log n.
\]
\end{theorem}

In particular, this shows that the upper bound for the expected gap of the averaging process on the torus is sharp. We conclude this work with a proof of a matching lower bound for the expected gap of the process on the cycle, whenever the loads are bounded away from zero.

\begin{theorem}
\label{thm:cycle-lower}
For every \(\delta>0\), there exists a constant \(c_\delta>0\) such that for every \(n\ge3\), every deterministic sequence \(0<\delta\le w_t\le1\), and every sufficiently large\footnote{depending on \(n\).} \(t\), the process on \(C_n\) satisfies
\[
\E[\Gap(t)]
\ge
c_\delta\sqrt n.
\]
\end{theorem}

In particular, this shows that the bounds for the expected gap of the averaging process on the cycle are sharp, thus proving the conjecture of Alistarh, Nadiradze, and Sabour~\cite{AlistarhNadiradzeSabour} that the expected gap is of order \(\sqrt n\).

\paragraph{Proof overview.}
We first prove the general upper bounds. Writing the centered update corresponding to a selected edge \(e\) as \( X(t+1)=A_eX(t)+w_tc_e \) gives a recursion for the moment generating function
\( M_t(\theta)=\E\exp\bigl(\ip{\theta}{X(t)}\bigr).\)
We define 
\[
R(\theta)=\frac d2\ip{\theta}{\mathcal L^+\theta}+\frac{\Re_G+2}{4}\|\theta\|_2^2.
\]
The main technical part of the proof is to show that, for every zero-sum vector \(\theta\) with sufficiently small \(\ell_1\)-norm,
\[
M_t(\theta)
\le
\exp\left(
4R(\theta)
\right).
\]
We first apply this bound with \(\theta=s(e_v-e_u)\), together with Chernoff's inequality, to obtain the pairwise tail bound of Theorem~\ref{thm:pairwise-tail}. Dirksen's mixed-tail generic-chaining theorem is then applied to these pairwise estimates, yielding Theorem~\ref{thm:general-upper}.

The proof of the moment generating function bound itself is given afterwards, and uses the effect of a single averaging step on the quadratic expression on the right.

\medskip

We next apply the general bound to the cycle and torus graphs. For the cycle a direct calculation of the effective resistance yields the \(O(\sqrt n)\) upper bound, and for the torus, the known estimate \(R_{\max}(G)=O(\log n)\) gives an \(O(\log n)\) bound. 

\medskip

The relation between \(\gamma_2\) and the cover time, together with the \(O(n^2)\) cover-time bound for regular graphs, gives that for every regular graph
\[
\E[\Gap(t)]=O_d(\sqrt n).
\]

Finally, for the cycle we use the lower bound of Alistarh, Nadiradze, and Sabour for the expected squared difference between opposite vertices, while Theorem~\ref{thm:pairwise-tail} gives a corresponding fourth-moment bound. The Paley--Zygmund inequality then yields an \(\Omega(\sqrt n)\) expected gap for unit arrivals, and a second-moment comparison extends the result to every deterministic sequence satisfying \(\delta\le w_t\le1\).

%\paragraph{Organization.} In Section~\ref{sec:general-proof} we prove Theorem~\ref{thm:pairwise-tail} and Theorem~\ref{thm:general-upper}. Namely, in~\ref{sec:2.2fromMGFBound} we show how Theorem~\ref{thm:pairwise-tail} follows from the moment generating function bound, and in~\ref{sec:2.3from2.2} we show how from Theorem~\ref{thm:pairwise-tail} follows Theorem~\ref{thm:general-upper}. In~\ref{sec:MGFBound} we prove the moment generating function bound. All the additional proofs appear in Section~\ref{sec:applications}

\section{Proof of the Main Results}
\label{sec:general-proof}

We now prove Theorem~\ref{thm:pairwise-tail} and Theorem~\ref{thm:general-upper}. Throughout this section, \(G=(V,E)\) is a fixed finite connected \(d\)-regular graph, \(n=|V|\).

\medskip

All constants denoted by \(c\) or \(C\) are universal unless another dependence is explicitly indicated, and the same notation may be used for different universal constants at different occurrences.

\subsection{Bounding the Moment Generating Function}

For an edge \(e=\{u,v\}\in E\), define
\[
A_e
:=
I-\frac12(e_u-e_v)(e_u-e_v)^{\mathsf T}
\]
and
\[
c_e
:=
\frac12(e_u+e_v)-\frac1n\1.
\]
The matrix \(A_e\) replaces the coordinates at \(u\) and \(v\) by their common average. Indeed, if \(\theta\in\R^V\), then
\[
(A_e\theta)_u=(A_e\theta)_v=\frac{\theta_u+\theta_v}{2},
\]
while every other coordinate is unchanged. Consequently, if the edge \(e\) is selected at time \(t\), then the centered load vector satisfies
\[
X(t+1)=A_eX(t)+w_tc_e.
\]

For \(\theta\in\R^V\), denote the moment generating function
\[
M_t(\theta)
=
\E\exp\bigl(\ip{\theta}{X(t)}\bigr).
\]
Since \(X(t)\) has zero sum, adding a multiple of \(\1\) to \(\theta\) does not change \(M_t(\theta)\). We shall therefore work throughout on the zero-sum subspace
\[
\1^\perp
=
\left\{
\theta\in\R^V:
\sum_{i\in V}\theta_i=0
\right\}.
\]

\begin{lemma}
\label{lem:mgf-recursion}
For every \(\theta\in\R^V\),
\[
M_{t+1}(\theta)
=
\frac1{|E|}
\sum_{e\in E}
\exp\bigl(w_t\ip{\theta}{c_e}\bigr)
M_t(A_e\theta).
\]
\end{lemma}

\begin{proof}
Condition on the edge \(e\) selected at time \(t\). Since \(A_e\) is symmetric,
\[
\ip{\theta}{X(t+1)}
=
\ip{\theta}{A_eX(t)+w_tc_e}
=
\ip{A_e\theta}{X(t)}
+
w_t\ip{\theta}{c_e}.
\]
Taking expectation and averaging over the uniform choice of \(e\) gives the stated recursion.
\end{proof}

For every zero-sum vector \(\theta\), define
\[
R(\theta)
:=
\frac d2\ip{\theta}{\mathcal L^+\theta}
+
\frac{\Re_G+2}{4}\|\theta\|_2^2.
\]
Both terms are nonnegative. The first is an effective resistance energy associated with the Laplacian. The second is a Euclidean correction term. Its coefficient is chosen so that, after one averaging step, the quadratic contribution coming from the effective resistance term is controlled uniformly over all edges, including those whose effective resistance is close to the maximum, appearing in the definition of \(\Re_G\).

\begin{proposition}
\label{prop:mgf-bound}
For every \(t\ge0\) and every \(\theta\in \mathbf 1 ^\perp\) satisfying
\(\|\theta\|_1 \le \frac{1}{3(\Re_G+1)},\) one has
\[
M_t(\theta)
\le
\exp\bigl(4R(\theta)\bigr).
\]
\end{proposition}

We postpone the proof of Proposition~\ref{prop:mgf-bound} and first derive the pairwise tail estimate and the bound on the expected gap.

\subsection{Pairwise Tail Bound}\label{sec:2.2fromMGFBound}

\begin{proof}[Proof of Theorem~\ref{thm:pairwise-tail}]
Fix \(u,v\in V\). Denote 
\(
r_{u,v}
=
dR_{\mathrm{eff}}(u,v)+\Re_G+2.
\)

\medskip 

Let \(0\le s\le\frac{1}{6(\Re_G+1)},\) and take \(\theta=s(e_v-e_u).\) Note that \(\theta\) has zero sum and
\(
\|\theta\|_1=2s
\le
\frac{1}{3(\Re_G+1)},
\) thus Proposition~\ref{prop:mgf-bound} may be applied. Moreover,
\[
\ip{\theta}{X(t)}=s(X_v(t)-X_u(t)), \quad \ip{\theta}{\mathcal L^+\theta}
=
s^2R_{\mathrm{eff}}(u,v), \quad \|\theta\|_2^2=2s^2.
\]
Therefore,
\[
\E e^{s(X_v(t)-X_u(t))}
\le
\exp\bigl(4R(\theta)\bigr)
=
\exp\left(
2ds^2R_{\mathrm{eff}}(u,v)
+
2(\Re_G+2)s^2
\right)
=
\exp\left(
2r_{u,v}s^2
\right).
\]
Chernoff's inequality now yields, for every \(0\le s\le\frac{1}{6(\Re_G+1)}\),
\[
\Pr(X_v(t)-X_u(t)\ge x)
\le
\exp\left(
-sx+2r_{u,v}s^2
\right).
\]

\medskip

We distinguish two ranges of \(x\):

\medskip

\emph{Case 1.} Suppose that \( x \le \frac{2r_{u,v}}{3(\Re_G+1)}.\)
We take \( s=\frac{x}{4r_{u,v}}\) (indeed this choice
satisfies \(s\le \frac 1 {6(\Re_G+1)}\)). Substituting this value in the inequality gives
\[
\Pr(X_v(t)-X_u(t)\ge x)
\le
\exp\left(
-\frac{x^2}{8r_{u,v}}
\right).
\]

\emph{Case 2.} Suppose that \( x > \frac{2r_{u,v}}{3(\Re_G+1)}.
\)
We take \( s=\frac{1}{6(\Re_G+1)},\) and notice
\[
-sx+2r_{u,v}s^2
=
-\frac{x}{6(\Re_G+1)}
+
\frac{r_{u,v}}{18(\Re_G+1)^2}<-\frac{x}{6(\Re_G+1)}
+ \frac{x}{12(\Re_G+1)}.
\]
Therefore
\[
\Pr(X_v(t)-X_u(t)\ge x)
\le
\exp\left(
-\frac{x}{12(\Re_G+1)}
\right).
\]
Applying exactly the same argument to \(X_u(t)-X_v(t)\) symmetrically, and then combining the two estimates, gives
\[
\Pr\left(
|X_v(t)-X_u(t)|\ge x
\right)
\le
2\exp\left(
-c\min\left\{
\frac{x^2}{r_{u,v}},
\frac{x}{\Re_G+1}
\right\}
\right)
\]
for a universal constant \(c>0\). This is precisely the asserted estimate.
\end{proof}

\subsection{Generic Chaining}

We use the following consequence of Dirksen's mixed-tail generic-chaining theorem~\cite[Theorem~3.5]{DirksenGenericChaining}.

\begin{lemma}
\label{lem:mixed-generic-chaining}
Let \((Y_v)_{v\in V}\) be a finite real-valued process, and let \(r_1,r_2\) be semimetrics on \(V\). Suppose that, for every distinct \(u,v\in V\) and every \(z\ge0\),
\[
\Pr\left(
|Y_v-Y_u|
\ge
\sqrt z\,r_2(u,v)+zr_1(u,v)
\right)
\le
2e^{-z}.
\]
Then, for every fixed \(o\in V\),
\[
\E\max_{v\in V}|Y_v-Y_o|
\le
C\left(
\gamma_2(V,r_2)+\gamma_1(V,r_1)
\right),
\]
where \(C\) is a universal constant.
\end{lemma}

\begin{proof}
By~\cite[Theorem~3.5]{DirksenGenericChaining}, applied with \(p=1\),
\[
\E\max_{v\in V}|Y_v-Y_o|
\le
C\left(
\gamma_2(V,r_2)+\gamma_1(V,r_1)
\right)
+
2\max_{v\in V}\E|Y_v-Y_o|.
\]
We show that the last term is already controlled by the two \(\gamma\)-functionals.
Fix \(v\in V\). If \(z\ge1\), then \(\sqrt z\le z\), and hence
\[
\sqrt z\,r_2(v,o)+zr_1(v,o)
\le
z(r_2(v,o)+r_1(v,o)).
\]
Consequently,
\[
\Pr\left(
|Y_v-Y_o|
\ge
z(r_2(v,o)+r_1(v,o))
\right)
\le
2e^{-z},
\qquad z\ge1.
\]
Thus,
\[
\begin{aligned}
\E|Y_v-Y_o|
&=
\int_0^\infty
\Pr\left(
|Y_v-Y_o|\ge x
\right)\,dx\\
&\le
(r_2(v,o)+r_1(v,o))
+
\int_{r_2(v,o)+r_1(v,o)}^\infty
\Pr\left(
|Y_v-Y_o|\ge x
\right)\,dx\\
&\le
(r_2(v,o)+r_1(v,o))
+
(r_2(v,o)+r_1(v,o))\int_1^\infty2e^{-z}\,dz\\
&\le
C(r_2(v,o)+r_1(v,o)).
\end{aligned}
\]
Taking the maximum over \(v\) gives
\[
\max_{v\in V}\E|Y_v-Y_o|
\le
C\left(
\operatorname{diam}(V,r_2)
+
\operatorname{diam}(V,r_1)
\right).
\]
Every admissible sequence satisfies \(\mathcal A_0=\{V\}\), thus for every \(v\in V\), the level \(k=0\) term in the defining sum is exactly \(\operatorname{diam}(V,r).\)
All the remaining terms are nonnegative, and therefore
\[
\operatorname{diam}(V,r)
\le
\gamma_\alpha(V,r)
\]
for every \(\alpha\ge1\). It follows that
\[
\max_{v\in V}\E|Y_v-Y_o|
\le
C\left(
\gamma_2(V,r_2)+\gamma_1(V,r_1)
\right).
\]
This proves the lemma.
\end{proof}

\subsection{Expected Gap Bound}\label{sec:2.3from2.2}

\begin{proof}[Proof of Theorem~\ref{thm:general-upper}]
By Theorem~\ref{thm:pairwise-tail}, after changing the universal constants, the hypothesis of Lemma~\ref{lem:mixed-generic-chaining} holds with
\[
r_2(u,v)
=
C\left(
\rho(u,v)
+
\sqrt{\Re_G+2}\,\mathbf 1_{\{u\ne v\}}
\right)
\]
and
\[
r_1(u,v)
=
C(\Re_G+1)\mathbf 1_{\{u\ne v\}}.
\]
To see this explicitly, set again
\( r_{u,v} = dR_{\mathrm{eff}}(u,v)+\Re_G+2\), and recall that Theorem~\ref{thm:pairwise-tail} gives
\[
\Pr\left(
|X_v(t)-X_u(t)|\ge x
\right)
\le
2\exp\left(
-c\min\left\{
\frac{x^2}{r_{u,v}},
\frac{x}{\Re_G+1}
\right\}
\right).
\]
Choose a sufficiently large universal constant \(C_0\). For
\[
x
=
C_0\left(
\sqrt{zr_{u,v}}
+
z(\Re_G+1)
\right),
\]
we have both
\[
\frac{x^2}{r_{u,v}}
\ge
C_0^2z, \qquad \text{and} \qquad
\frac{x}{\Re_G+1}
\ge
C_0z.
\]

Thus, we may choose \(C_0\) sufficiently large, depending only on the universal constant \(c\), such that \(\Pr\left(|X_v(t)-X_u(t)|\ge x \right) \le 2e^{-z}.\) 

Also, we notice that \(\sqrt{r_{u,v}}\le \rho(u,v)+\sqrt{\Re_G+2}\,\mathbf 1_{\{u\ne v\}},\) which gives the displayed choice for \(r_2\).

\medskip

We now estimate their chaining functionals. Let \(k_0\) be the smallest integer such that
\(2^{2^{k_0}}\ge n.\) Consider the admissible sequence that uses the trivial partition \(\mathcal A_k = \{V\}\) for \(k<k_0\) and the partition into singletons \(\mathcal A _k = \{\{v\}:v\in V\}\) for \(k\ge k_0\). Consider the discrete metric $\mathbf 1_{\{u\neq v\}}(u,v)$, and notice that every singleton set has \(\mathbf 1_{\{u\neq v\}}\)-diameter equal to zero, and every non singleton set has diameter one. Therefore,
\[
\gamma_1(V,\mathbf 1_{\{u\neq v\}})
\le
\sum_{k=0}^{k_0-1}2^k
\le
C\log(2n),
\]
and
\[
\gamma_2(V,\mathbf 1_{\{u\neq v\}})
\le
\sum_{k=0}^{k_0-1}2^{k/2}
\le
C\sqrt{\log(2n)}.
\]
Using homogeneity of the \(\gamma\)-functionals and their subadditivity under sums of semimetrics, we obtain, up to a universal multiplicative constant,
\[
\gamma_1(V,r_1)
\le
C(\Re_G+1)\log(2n)
\]
and
\[
\gamma_2(V,r_2)
\le
C\left(
\gamma_2(V,\rho)
+
\sqrt{(\Re_G+2)\log(2n)}
\right).
\]
Therefore, for every fixed \(o\in V\), Lemma~\ref{lem:mixed-generic-chaining} gives
\[
\E\max_{v\in V}|X_v(t)-X_o(t)|
\le
C\left(
\gamma_2(V,\rho)
+
(\Re_G+1)\log(2n)
\right).
\]
Since \(
\Gap(t)
\le
2\max_{v\in V}|X_v(t)-X_o(t)|,
\) the first assertion follows.

\medskip

For the resistance-diameter estimate, apply the same admissible partition construction directly to the metric \(\rho\). It gives the elementary cardinality bound
\[
\gamma_2(V,\rho)
\le
C\operatorname{diam}(V,\rho)\sqrt{\log(2n)}.
\]
By the definition of \(\rho\),
\[
\operatorname{diam}(V,\rho)
=
\sqrt{dR_{\max}(G)}.
\]
Substituting this into the first assertion gives
\[
\E[\Gap(t)]
\le
C\left(
\sqrt{dR_{\max}(G)\log(2n)}
+
(\Re_G+1)\log(2n)
\right),
\]
as claimed.
\end{proof}

\subsection{Proof of the Moment Generating Function Bound}\label{sec:MGFBound}

We begin with an auxiliary Lemma.

\begin{lemma}
\label{lem:voltage-bound}
Let \(\theta\in \mathbf 1^{\perp}\), and set \(\psi=\mathcal L^+\theta.\)
Then, for every edge \(\{u,v\}\in E\),
\[
|\psi_u-\psi_v|
\le
\frac{\|\theta\|_1}{2}R_{\mathrm{eff}}(u,v)
\]
\end{lemma}

\begin{proof}
Fix an orientation of the edge \(\{u,v\}\), and define
\[
g
=
\mathcal L^+(e_u-e_v).
\]
The identity \(\operatorname{osc}(g)=R_{\mathrm{eff}}(u,v)\) is standard, nevertheless, we include a short derivation for completeness.
 Because \(\mathcal L\) is real and symmetric, its Moore--Penrose pseudoinverse \(\mathcal L^+\) is symmetric as well. Moreover, since \(G\) is connected, \(\mathcal L\mathcal L^+\) is the orthogonal projection onto \(\1^\perp\). Consequently,
\[
\mathcal Lg
=
\mathcal L\mathcal L^+(e_u-e_v)
=
e_u-e_v.
\]
Thus, at every vertex $i \in V$ other than $u$ and $v$ we get, \(dg_i - \sum_{j\sim i} g_j = (\mathcal L g)_i = 0\), so \(g\) is harmonic at every vertex other than \(u\) and \(v\). 
By the discrete maximum principle (see \cite[Chapter 2]{LyonsPeres} for example), the maximum and minimum of \(g\) are attained at \(u\) and \(v\), that is,
\[
\max_i g_i=g_u,
\qquad
\min_i g_i=g_v.
\]
Therefore,
\[
\begin{aligned}
\operatorname{osc}(g)
&=
g_u-g_v\\
&=
(e_u-e_v)^{\mathsf T}
\mathcal L^+
(e_u-e_v)\\
&=
R_{\mathrm{eff}}(u,v).
\end{aligned}
\]
Write \(\theta=\theta^+-\theta^-\) for the positive and negative parts of \(\theta\). Using the symmetry of \(\mathcal L^+\), we obtain
\[
|\psi_u-\psi_v|
=
|\ip{e_u-e_v}{\mathcal L^+\theta}|=
|\ip{\mathcal L^+(e_u-e_v)}{\theta}|=
|\ip{g}{\theta}|=
\]
\[ 
\left|\sum_i\theta_i^+g_i-\sum_i\theta_i^-g_i\right| \underset{(*)}{\le}
\frac{\|\theta\|_1}{2}\operatorname{osc}(g)=
 \frac{\|\theta\|_1}{2}R_{\mathrm{eff}}(u,v).
\]
Where the inequality marked $(*)$ follows by the fact that both sums are convex combination of entries of $g$ multiplied by a scale of $\frac{\|\theta\|_1}{2}$, and thus lie in $[\min_i g_i , \max_i g_i].$
\end{proof}

\begin{proposition}
\label{prop:one-step}
Let \(G\) be a connected \(d\)-regular graph. For every zero-sum vector \(\theta\in\ \mathbf 1 ^\perp \) satisfying \(\|\theta\|_1 \le \frac{1}{3(\Re_G+1)}\)
and every \(w\in[0,1]\),
\[
\frac1{|E|}
\sum_{e\in E}
\exp\left(
 w\ip{\theta}{c_e}
+
4\bigl(R(A_e\theta)-R(\theta)\bigr)
\right)
\le1.
\]
\end{proposition}

\begin{proof}
We again denote \(
\psi=\mathcal L^+\theta.\)
Fix an orientation of every edge \(e=\{u,v\}\) and define
\[
\delta_e=\theta_u-\theta_v,
\qquad
z_e=\ip{\theta}{c_e}, \qquad q_e
=
R(A_e\theta)-R(\theta).
\]
Since \(\theta\) has zero sum,
\[
z_e
=
\frac{\theta_u+\theta_v}{2}.
\]
We compute now \(q_e\). Since
\(
A_e\theta
=
\theta-\frac{\delta_e}{2}(e_u-e_v),
\)
the effective resistance part of \(R\) changes by
\[
\begin{aligned}
&\frac d2
\left(
\ip{A_e\theta}{\mathcal L^+A_e\theta}
-
\ip{\theta}{\mathcal L^+\theta}
\right)\\
&\qquad=
-\frac d2\delta_e
\ip{e_u-e_v}{\mathcal L^+\theta}
+
\frac d8\delta_e^2
\ip{e_u-e_v}{\mathcal L^+(e_u-e_v)}\\
&\qquad=
-\frac d2\delta_e(\psi_u-\psi_v)
+
\frac d8\delta_e^2R_{\mathrm{eff}}(u,v).
\end{aligned}
\]
On the other hand,
\[
\begin{aligned}
\|A_e\theta\|_2^2-\|\theta\|_2^2
&=
\left\|
\theta-\frac{\delta_e}{2}(e_u-e_v)
\right\|_2^2
-
\|\theta\|_2^2\\
&=
-\delta_e\ip{\theta}{e_u-e_v}
+
\frac{\delta_e^2}{4}\|e_u-e_v\|_2^2\\
&=
-\delta_e^2+
\frac12\delta_e^2\\
&=
-\frac12\delta_e^2.
\end{aligned}
\]
Multiplying by the coefficient \((\Re_G+2)/4\) in the definition of \(R\), we obtain
\[
q_e
=
-\frac d2\delta_e(\psi_u-\psi_v)
+
\frac{dR_{\mathrm{eff}}(u,v)-(\Re_G+2)}{8}\delta_e^2 
\le
-\frac d2\delta_e(\psi_u-\psi_v)
-
\frac{1}{4}\delta_e^2.
\]
We next sum the first-order terms over all edges. We observe that
\[
\sum_{e\in E}z_e
=
\frac12
\sum_{\{u,v\}\in E}
(\theta_u+\theta_v)
=
\frac d2\sum_{i\in V}\theta_i
=
0
\]
and also,
\[
\sum_{e\in E}z_e^2
=
\frac14
\sum_{\{u,v\}\in E}
(\theta_u+\theta_v)^2
=
\frac d2\|\theta\|_2^2
-
\frac14\ip{\theta}{\mathcal L\theta}.
\]
To calculate the mixed term in \(q_e\), note that
\(
\sum_{e=\{u,v\}\in E} \delta_{e}(\psi_u-\psi_v)=\ip{\theta}{\mathcal L\psi}=\ip{\theta}{\mathcal L \mathcal L^+ \theta}=\|\theta\|_2^2.\) From the previous bound on $q_e$ we obtain
\[
\sum_{e\in E}q_e
\le
\sum_{e=\{u,v\}\in E} -\frac d2  \delta_{e}(\psi_u-\psi_v)
-
\sum_{e\in E} \frac 1 4\delta_e^2
=
-\frac d2\|\theta\|_2^2
-
\frac14\ip{\theta}{\mathcal L\theta}
=
-
\sum_{e\in E}z_e^2
-
\frac12\ip{\theta}{\mathcal L\theta}.
\]
Since $\theta \in \mathbf 1^\perp$, the negative and positive entries of $\theta$ have equal sum, and in particular $\|\theta^-\|_1=\|\theta^+\|_1 = \frac{\|\theta\|_1}{2}$. Thus \(|\delta_e| \le \|\theta\|_1\) for every edge \(e\), and by Lemma~\ref{lem:voltage-bound},
\(
|\psi_u-\psi_v|\le
\frac{\|\theta\|_1}{2}R_{\mathrm{eff}}(u,v)
\le
\frac{\Re_G\|\theta\|_1}{2d}.
\)
Using again the previous calculation of $q_e$ we see
\[
\begin{aligned}
|q_e|
&\le
\frac d2|\delta_e|\,|\psi_u-\psi_v|
+
\frac{(\Re_G+2)-dR_{\mathrm{eff}}(u,v)}{8}\delta_e^2\\
&\le
\frac{\Re_G\|\theta\|_1}{4}|\delta_e|
+
\frac{\Re_G+2}{8}\delta_e^2\\
&\le
\frac{\Re_G\|\theta\|_1}{4}|\delta_e|
+
\frac{\Re_G+2}{8}\|\theta\|_1|\delta_e|\\
&=
\frac{3\Re_G+2}{8}
\|\theta\|_1|\delta_e|.
\end{aligned}
\]
Then
\[
\sum_{e\in E}q_e^2
\le
\left(\frac{3\Re_G+2}{8}\right)^2
\|\theta\|^2_1
\sum_{e\in E}\delta_e^2
=
\left(\frac{3\Re_G+2}{8}\right)^2
\|\theta\|^2_1
\ip{\theta}{\mathcal L\theta}.
\]

We are now ready for the exponential estimate. Denote \(y_e = wz_e+4q_e.\) Notice
\[
\sum_{e\in E}y_e
=
 w\sum_{e\in E}z_e
+
4\sum_{e\in E}q_e
\le
-4\sum_{e\in E}z_e^2
-
2\ip{\theta}{\mathcal L\theta}.
\]
Since \((x+y)^2\le2x^2+2y^2\) we also get,
\[
\sum_{e\in E}y_e^2
\le
2w^2\sum_{e\in E}z_e^2
+
32\sum_{e\in E}q_e^2
\le
2\sum_{e\in E}z_e^2
+32
\left(\frac{3\Re_G+2}{8}\right)^2
\|\theta\|^2_1
\ip{\theta}{\mathcal L\theta}.
\]
Since \(\|\theta\|_1 \le \frac 1{3(\Re_G+1)}\) we have 
\(
32 \left(\frac{3\Re_G+2}{8}\right)^2
\|\theta\|^2_1\le \frac{ (\Re_G+\frac23)^2}{2(\Re_G+1)^2}\le \frac12
\)
and consequently,
\[
\sum_{e\in E}(y_e+y_e^2)
\le
(-4+2)\sum_{e\in E}z_e^2
+
\left(
-2+\frac{1}{2}
\right)
\ip{\theta}{\mathcal L\theta}
\le
0.
\]
Note that
\[
|y_e|
\le
|z_e|+4|q_e|
\le
\|\theta\|_1+\frac{3\Re_G+2}{2} \|\theta\|_1^2
\le
\frac32 \|\theta\|_1\le
\frac12.
\]
For \(|y|\le \frac12\), the inequality \( e^y\le1+y+y^2 \) holds, and by applying it to every \(y_e\) and summing over the edges gives
\[
\begin{aligned}
\sum_{e\in E}e^{y_e}
&\le
\sum_{e\in E}(1+y_e+y_e^2)\\
&=
|E|+
\sum_{e\in E}(y_e+y_e^2)\\
&\le
|E|.
\end{aligned}
\]
Dividing by \(|E|\) proves the proposition.
\end{proof}

\medskip

We are now ready to prove the moment generating function bound.

\medskip

\begin{proof}[Proof of Proposition~\ref{prop:mgf-bound}]
We proceed by induction on \(t\). At time \(0\), \(X(0)=0\), and the claim follows trivially
\(
M_0(\theta)=1\le\exp\bigl(4R(\theta)\bigr).
\)

\medskip

Assume the result holds at time \(t\). For every edge \(e=\{u,v\}\), the map \(A_e\) preserves the sum of the coordinates and also contracts the \(\ell_1\) norm. Indeed, it changes only the coordinates at \(u\) and \(v\), and
\[
|(A_e\theta)_u|+|(A_e\theta)_v|
=
2\left|
\frac{\theta_u+\theta_v}{2}
\right|
=
|\theta_u+\theta_v|
\le
|\theta_u|+|\theta_v|.
\]
Thus, whenever \(\theta\) is zero-sum and \( \|\theta\|_1 \le \frac{1}{3(\Re_G+1)}, \) the vector \(A_e\theta\) also satisfies both conditions. By Lemma~\ref{lem:mgf-recursion} and the induction hypothesis,
\[
\begin{aligned}
M_{t+1}(\theta)
&=
\frac1{|E|}
\sum_{e\in E}
\exp\bigl(w_t\ip{\theta}{c_e}\bigr)
M_t(A_e\theta)\\
&\le
\frac1{|E|}
\sum_{e\in E}
\exp\left(
 w_t\ip{\theta}{c_e}
+
4R(A_e\theta)
\right)\\
&=
\exp\bigl(4R(\theta)\bigr)
\frac1{|E|}
\sum_{e\in E}
\exp\left(
 w_t\ip{\theta}{c_e}
+
4\bigl(R(A_e\theta)-R(\theta)\bigr)
\right)\\
&\le
\exp\bigl(4R(\theta)\bigr),
\end{aligned}
\]
where the final inequality follows from Proposition~\ref{prop:one-step}.
\end{proof}

\section{Additional Proofs}
\label{sec:applications}

\subsection{The Cycle and the Torus}\label{sec:cycleandTorus}

\begin{proof}[Proof of Corollary~\ref{cor:cycle-upper}]
Let \(G=C_n\) with $V=\mathbb Z / n\mathbb Z$. Then \(d=2\), and the cycle is edge-transitive, thus every edge has the same effective resistance. Hence, as in Remark~\ref{rem:sameER} we have
\( \Re_{C_n} = \frac{2(n-1)}{n}\).

Fix distinct vertices \(i,j\in C_n\). Note that
\[
R_{\mathrm{eff}}(i,j)
=
\frac{|i-j|(n-|i-j|)}{n}.
\] 
In particular,
\[
\rho(i,j)  =
\sqrt{2R_{\mathrm{eff}}(i,j)}
\le
\sqrt{2\operatorname{dist}_{C_n}(i,j)}.
\]
We construct an admissible sequence of partitions by consecutive arcs of the cycle. Take \(\mathcal A_0 = \{V\}\). For \(k>0\), put \(m_k=\min\{n,2^{2^k}\},\) and partition the vertices into at most \(m_k\) consecutive arcs along the cycle of \say{almost} even size (i.e. whose size differ by at most one). denote this admissible partition by \(\mathcal A_k\). 

If \(A\in\mathcal A_k\) is not a singleton, every two vertices in \(A\) are joined inside that arc by a path of length at most
\( \left\lceil\frac{n}{m_k}\right\rceil. \) Consequently,
\[
\operatorname{diam}(A,\rho)
\le
\sqrt{
2\left\lceil\frac{n}{m_k}\right\rceil
}.
\]
Once \(m_k=n\), all sets may be chosen to be singletons and their diameter is zero. Thus, for every vertex \(i\),
\[
\sum_{k\ge0}
2^{k/2}
\operatorname{diam}(A_k(i),\rho)
\le
C\sqrt n
\sum_{k\ge0}
2^{k/2}2^{-2^{k-1}}
\le
C'\sqrt n.
\]
In particular,
\(\gamma_2(C_n,\rho) \le C\sqrt n,\) thus by Theorem~\ref{thm:general-upper},
\[
\E[\Gap(t)]
=
O\left(
\sqrt n+\log(2n)
\right)
=
O(\sqrt n)
\]
for every \(n\ge3\), every \(t\ge0\), and every deterministic arrival sequence \(0\le w_t\le1\). This proves the corollary.
\end{proof}

\begin{proof}[Proof of Corollary~\ref{cor:torus}]
Let
\(G=\mathbb T_m^2=\mathbb Z/m\mathbb Z \, \times \mathbb Z/m\mathbb Z,
\qquad
n=m^2.\)
Again, by remark~\ref{rem:sameER} we have \( \Re_{\mathbb T_m^2}=\frac{2(n-1)}{n}.\)
Chandra, Raghavan, Ruzzo, Smolensky, and Tiwari~\cite[Theorem~6.1]{ChandraResistance} show that
\[
R_{\max}(G)
=
O(\log n).
\]
Applying Theorem~\ref{thm:general-upper} and using \(\Re_G<2\), we obtain
\[
\E[\Gap(t)]
=
O(\log(n))
\]
uniformly in \(t\) and in all deterministic arrival sequences \(0\le w_t\le1\). 
\end{proof}

\subsection{Bounds for Regular Graphs}\label{sec:BoundsRegGraphs}

\begin{proof}[Proof of Theorem~\ref{thm:bounded-degree-regular}]
Let \(t_{\mathrm{cov}}(G)\) denote the cover time of the simple random walk on \(G\) (see~\cite{LevinPeres} for example). Ding, Lee, and Peres~\cite[Theorem~1.2]{DingLeePeres} show that
\[
\gamma_2
\left(
V,\sqrt{R_{\mathrm{eff}}}
\right)
\le
C
\sqrt{
\frac{t_{\mathrm{cov}}(G)}{|E|}
}.
\]
In particular,
\[
\begin{aligned}
\gamma_2(V,\rho)
&=
\gamma_2
\left(
V,\sqrt{dR_{\mathrm{eff}}}
\right)\\
&=
\sqrt d\,
\gamma_2
\left(
V,\sqrt{R_{\mathrm{eff}}}
\right)\\
&\le
C
\sqrt{
\frac{d\,t_{\mathrm{cov}}(G)}
{|E|}
}.
\end{aligned}
\]
Because \(G\) is \(d\)-regular,
\( \gamma_2(V,\rho) \le C \sqrt{
\frac{t_{\mathrm{cov}}(G)}{n}}.\)
Kahn, Linial, Nisan, and Saks~\cite{KahnLinialNisanSaks} proved that the cover time of every \(n\)-vertex regular graph satisfies
\[
t_{\mathrm{cov}}(G)
\le
Cn^2.
\]
It follows that \( \gamma_2(V,\rho) \le C\sqrt n.\)
Also, recall from Remark~\ref{rem:sameER} that \(\Re_G\le d.\)
Substituting into Theorem~\ref{thm:general-upper} yields
\[
\E[\Gap(t)]
=
O\left(
\gamma_2(V,\rho)
+
(\Re_G+1)\log(2n)
\right)
=
O\left(
\sqrt n
+
(d+1)\log(2n)
\right)
=
O_d(\sqrt n).
\]
\end{proof}

\medskip

\begin{proof}[Proof of Theorem~\ref{thm:main-lower}]
We first consider unit arrivals, \(w_t\equiv1\). The proof is identical to the \say{folklore} coupon-collector lower-bound argument of Bansal and Feldheim~\cite{BansalFeldheim}. Namely, for a suitable time \(t_n=\Theta(n\log n)\), with probability $\Omega(1)$, there exists a vertex that has not been incident to any selected edge during the first \(t_n\) rounds. Such a vertex still has load \(0\), whereas the average load is \(t_n/n=\Theta(\log n)\), and hence \(\Gap(t_n)=\Omega(\log n)\) on this event.

The same argument applies whenever the arrivals satisfy \(w_t\ge\delta>0\). Indeed, on the event that some vertex remains untouched, that vertex still has load \(0\), while the total load at time \(t_n\) is at least \(\delta t_n\), and therefore
\(
\sup_{t\ge0}\E[\Gap(t)]\ge \E[\Gap(t_n)] \ge\delta t_n/n=\Omega_\delta(\log n).
\)
\end{proof}

\subsection{Sharpness of Bound for Cycle}\label{sec:LBoundCycle}

\begin{proof}[Proof of Theorem~\ref{thm:cycle-lower}]
We divide the proof into two cases:
\\
\\
\textbf{Unit-load case:} We first consider the case where all the loads introduced in the process are constant, i.e. $w_t = 1$ for all $t\ge 0$. Recall the definition of the \(\lfloor n/2\rfloor\)-hop potential which was introduced in~\cite{AlistarhNadiradzeSabour},
\[
\phi_{\lfloor n/2\rfloor}(t)
=
\sum_{i=0}^{n-1}
\bigl(X_i(t)-X_{i+{\lfloor n/2\rfloor}}(t)\bigr)^2.
\]
The lower-bound argument given in~\cite{AlistarhNadiradzeSabour}, when applied to this case\footnote{The assumption of the authors on the loads $w_t$ is that they are generated randomly by the same distribution $W$ satisfying $\E(W^2)=1$. We may see that their assumption applies when considering the case where $w_t=1$ for all $t \ge 0$}, shows that for all sufficiently large \(t\),
\[
\E[\phi_{\lfloor n/2\rfloor}(t)]\ge c n^2.
\]
By rotational symmetry, the expected squared difference is the same for every such pair. Therefore,  
\[
\E\bigl[(X_0(t)-X_{\lfloor n/2\rfloor}(t))^2\bigr]
=
\frac1n\E[\phi_{\lfloor n/2\rfloor}(t)]
\ge cn.
\]
Theorem~\ref{thm:pairwise-tail} implies, after changing universal constants, that
\[
\Pr(|X_0(t)-X_{\lfloor n/2\rfloor}(t)|\ge u)
\le
C\exp\left(
-c\min\left\{\frac{u^2}{n},u\right\}
\right).
\]
Thus,
\[
\E|X_0(t)-X_{\lfloor n/2\rfloor}(t)|^4
=
4\int_0^\infty u^3\Pr(|X_0(t)-X_{\lfloor n/2\rfloor}(t)|\ge u)\,du
\le
\]
\[
4C\int_0^n u^3e^{-cu^2/n}\,du
+
4C\int_n^\infty u^3e^{-cu}\,du
\le
C'n^2.
\]

We now apply the Paley-Zygmund inequality~\cite{PaleyZygmund} to the nonnegative random
variable \((X_0(t)-X_{\lfloor n/2\rfloor}(t))^2\). Since
\(\E[(X_0(t)-X_{\lfloor n/2\rfloor}(t))^2]\ge cn\) and \(\E[(X_0(t)-X_{\lfloor n/2\rfloor}(t))^4]\le C'n^2\),
\[
\Pr\left(
(X_0(t)-X_{\lfloor n/2\rfloor}(t))^2\ge\frac12\E[(X_0(t)-X_{\lfloor n/2\rfloor}(t))^2]
\right)
\ge
\frac14
\frac{\E[(X_0(t)-X_{\lfloor n/2\rfloor}(t))^2]^2}{\E[(X_0(t)-X_{\lfloor n/2\rfloor}(t))^4]}
\ge \frac{c^2}{4C'}.
\]
Hence
\[
\Pr\left(
|X_0(t)-X_{\lfloor n/2\rfloor}(t)|\ge \sqrt \frac{cn}{2}
\right)
\ge \frac{c^2}{4C'}.
\]
Finally, \(\Gap(t)\ge |X_0(t)-X_{\lfloor n/2\rfloor}(t)|\), and therefore 
\[
\E[\Gap(t)]
\ge \sqrt{\frac{cn}{2}} \Pr\left(\Gap(t)\ge \sqrt{\frac{cn}{2}}\right) \ge
C''\sqrt n.
\]
As wanted.
\\
\\
\textbf{General case.}
We now assume that
\(
\delta \le w_t \le 1
\)
for every \(t \ge 0\). Let \(X^{(w)}(t)\) denote the centered load vector of this process, and let \(X^{(1)}(t)\) denote the centered load vector of the unit-load process.

For a vector \(\theta \in \R^n\), define
\[
F_t^{(w)}(\theta)
=
\E\left[ \ip{\theta}{X^{(w)}(t)} ^2
\right],
\]
and \(F_t^{(1)}(\theta)\) analogously. We claim that \(
F_t^{(w)}(\theta)
\ge
\delta^2 F_t^{(1)}(\theta),
\) for every \(t \ge 0, \ \theta \in \R^n\).

We prove this by induction on \(t\). It is immediate at \(t=0\), since both processes start from the zero configuration. 
Denote $A_j = A_{\{j,j+1\} } , c_j = c_{\{j,j+1\}}$.
By rotational symmetry and since both centered processes have zero sum, we have \(\E[X^{(w)}(t)]=\E[X^{(1)}(t)]=0\). Hence the cross terms vanish, and the update rule gives
\[
F_{t+1}^{(w)}(\theta)
=
\frac{1}{n}\sum_{j=0}^{n-1}
F_t^{(w)}(A_j\theta)
+
\frac{w_t^2}{n}
\sum_{j=0}^{n-1}
\ip{\theta}{c_j}^2 \ge \frac{\delta^2}{n}
\sum_{j=0}^{n-1}
F_t^{(1)}(A_j\theta)
+
\frac{\delta^2}{n}
\sum_{j=0}^{n-1}
\ip{\theta}{c_j}^2 \\
=
\delta^2 F_{t+1}^{(1)}(\theta).
\]

Taking
\(
\theta = e_i-e_{i+\lfloor n/2\rfloor}
\)
gives
\[
\E\left[
\left(
X_i^{(w)}(t)-X_{i+\lfloor n/2\rfloor}^{(w)}(t)
\right)^2
\right]
\ge
\delta^2
\E\left[
\left(
X_i^{(1)}(t)-X_{i+\lfloor n/2\rfloor}^{(1)}(t)
\right)^2
\right].
\]
Summing over \(i\), we conclude that
\(
\E\left[
\phi_{\lfloor n/2\rfloor}^{(w)}(t)
\right]
\ge
\delta^2
\E\left[
\phi_{\lfloor n/2\rfloor}^{(1)}(t)
\right].
\)
Thus, from the unit-load lower bound, for all sufficiently large \(t\),
\[
\E\left[
\phi_{\lfloor n/2\rfloor}^{(w)}(t)
\right]
\ge
c\delta^2 n^2.
\]

Repeating the Paley--Zygmund argument from the unit-load case, using the uniform fourth-moment bound supplied by Theorem~\ref{thm:pairwise-tail}, yields
\[
\E[\Gap(t)]
\ge
c_\delta \sqrt{n}.
\]
This proves the result in the general case. 
\end{proof}

\section{Further Discussion}
The results above suggest several natural extensions of the model. First, it would be interesting to understand to what extent the regularity assumption can be removed. Regularity is used at several points in the argument, most notably in the averaging over edges and in the normalization of the effective-resistance metric. For a general graph, the vertices are sampled with probabilities proportional to their degrees, and the natural centering and the corresponding quadratic form would therefore have to reflect the stationary measure of the underlying random walk. It seems plausible that an analogous pairwise concentration estimate should hold with an appropriately weighted Laplacian and resistance metric.

\medskip

A related extension is to allow a nonuniform choice of the averaging edge. More generally, suppose that each edge \(e\) is selected according to a fixed probability \(p_e\). This may be viewed as assigning conductances to the edges, and suggests replacing the usual graph Laplacian by the corresponding weighted Laplacian. Since effective resistance and generic chaining are naturally defined in this weighted setting, one may expect the same general approach to apply. The main question is whether the one-step moment generating function estimate can be adapted with constants controlled by suitable local parameters of the weighted graph.

\medskip

Another natural question is to determine more precisely which features of the resistance metric govern the gap. Our general upper bound is expressed through the generic-chaining functional
\(\gamma_2(V,\rho)\), while the lower bounds obtained here use different arguments. It would be interesting to identify graph families for which \(\gamma_2(V,\rho)\) gives the correct order of the expected gap, and more generally to understand whether a matching lower bound can be formulated directly in terms of the effective-resistance metric.

\medskip

Finally, while our main results concern the expected gap, the pairwise estimate of Theorem~\ref{thm:pairwise-tail} already provides exponential tail bounds for individual differences. It would be interesting to investigate whether the chaining argument can be strengthened to obtain sharp high-probability bounds for the gap itself, uniformly in time. We leave the development of the extensions above to future work.

\medskip

The upper bound also extends immediately to arbitrary initial configurations. Indeed, by linearity, the process can be decomposed into the sum of the process started from the zero configuration and the stationary averaging process applied to the initial load vector. Since every averaging step can only decrease the gap of the latter,
\[
\E[\Gap(t)]
\le
\Gap(0)+C\sqrt n.
\]
An improved bound that captures the decay of the contribution of the initial configuration would depend on the mixing behavior of the averaging process~\cite{SauerwaldSun}.

\paragraph{Acknowledgments.}
The author thanks Ron Peretz for introducing the problem and for many helpful discussions and valuable comments.


\begin{thebibliography}{99}

\bibitem{AldousLanoue}
D.~Aldous and D.~Lanoue,
A lecture on the averaging process,
\emph{Probability Surveys}
\textbf{9} (2012), 90--102.

\bibitem{AlistarhNadiradzeSabour}
D.~Alistarh, G.~Nadiradze, and A.~Sabour,
Dynamic averaging load balancing on cycles,
\emph{Algorithmica}
\textbf{84} (2022), 1007--1029.

\bibitem{AzarBroderKarlinUpfal}
Y.~Azar, A.~Z.~Broder, A.~R.~Karlin, and E.~Upfal,
Balanced allocations,
\emph{SIAM Journal on Computing}
\textbf{29} (1999), no.~1, 180--200.

\bibitem{BansalFeldheim}
N.~Bansal and O.~N.~Feldheim,
The power of two choices in graphical allocation,
\emph{SIAM Journal on Computing} (2024),
Special Section STOC 2022.

\bibitem{BerenbrinkCzumajStegerVocking}
P.~Berenbrink, A.~Czumaj, A.~Steger, and B.~V\"ocking,
Balanced allocations: The heavily loaded case,
\emph{SIAM Journal on Computing}
\textbf{35} (2006), no.~6, 1350--1385.

\bibitem{BerenbrinkHintzeHosseinpourKaaserRau}
P.~Berenbrink, L.~Hintze, H.~Hosseinpour, D.~Kaaser, and M.~Rau,
Dynamic averaging load balancing on arbitrary graphs,
in \emph{50th International Colloquium on Automata, Languages, and Programming (ICALP 2023)},
Leibniz International Proceedings in Informatics,
vol.~261, Schloss Dagstuhl--Leibniz-Zentrum f\"ur Informatik
(2023), pp.~18:1--18:18.

\bibitem{ChandraResistance}
A.~K.~Chandra, P.~Raghavan, W.~L.~Ruzzo, R.~Smolensky, and P.~Tiwari,
The electrical resistance of a graph captures its commute and cover times,
\emph{Computational Complexity}
\textbf{6} (1996), no.~4, 312--340.

\bibitem{DingLeePeres}
J.~Ding, J.~R.~Lee, and Y.~Peres,
Cover times, blanket times, and majorizing measures,
\emph{Annals of Mathematics}
\textbf{175} (2012), no.~3, 1409--1471.

\bibitem{DirksenGenericChaining}
S.~Dirksen,
Tail bounds via generic chaining,
\emph{Electronic Journal of Probability}
\textbf{20} (2015), no.~53, 1--29.

\bibitem{KahnLinialNisanSaks}
J.~D.~Kahn, N.~Linial, N.~Nisan, and M.~E.~Saks,
On the cover time of random walks on graphs,
\emph{Journal of Theoretical Probability}
\textbf{2} (1989), no.~1, 121--128.

\bibitem{KleinRandic}
D.~J.~Klein and M.~Randi\'c,
Resistance distance,
\emph{Journal of Mathematical Chemistry}
\textbf{12} (1993), 81--95.

\bibitem{LevinPeres}
D.~A.~Levin and Y.~Peres,
\emph{Markov Chains and Mixing Times},
2nd ed., with contributions by E.~L.~Wilmer,
American Mathematical Society, Providence, RI (2017).

\bibitem{LyonsPeres}
R.~Lyons and Y.~Peres,
\emph{Probability on Trees and Networks},
Cambridge Series in Statistical and Probabilistic Mathematics, vol.~42,
Cambridge University Press (2016).

\bibitem{PaleyZygmund}
R.~E.~A.~C.~Paley and A.~Zygmund,
A note on analytic functions in the unit circle,
\emph{Mathematical Proceedings of the Cambridge Philosophical Society}
\textbf{28} (1932), no.~3, 266--272.

\bibitem{PeresTalwarWieder}
Y.~Peres, K.~Talwar, and U.~Wieder,
Graphical balanced allocations and the \((1+\beta)\)-choice process,
\emph{Random Structures \& Algorithms}
\textbf{47} (2015), no.~4, 760--775.

\bibitem{SauerwaldSun}
T.~Sauerwald and H.~Sun,
Tight bounds for randomized load balancing on arbitrary network topologies,
in \emph{Proceedings of the 53rd Annual IEEE Symposium on Foundations of Computer Science},
IEEE, (2012), pp.~341--350.

\bibitem{TalagrandOrig}
M.~Talagrand,
\emph{Regularity of Gaussian processes},
Acta Math. \textbf{159} (1987), no.~1--2, 99--149.

\bibitem{TalagrandGamma}
M.~Talagrand,
\emph{Majorizing measures: the generic chaining},
Ann. Probab. \textbf{24} (1996), no.~3, 1049--1103.

\end{thebibliography}
\end{document}